\documentclass[11pt]{article}
\usepackage[english]{babel}
\usepackage[utf8]{inputenc}
\usepackage{amsmath}
\usepackage{amsfonts}
\usepackage{amsthm}
\usepackage{amssymb}
\usepackage{lmodern}
\usepackage{color}
\usepackage{tabularx}
\usepackage{graphicx}
\usepackage{subfigure}
\usepackage{placeins}

\usepackage[a4paper,left=3cm,right=3cm,top=2cm,bottom=4cm,bindingoffset=5mm]{geometry}

\newtheorem{theorem}{Theorem}
\newtheorem{proposition}{Proposition}
\newtheorem{lemma}{Lemma}
\newtheorem{corollary}{Corollary}
\newtheorem{remark}{Remark}

\begin{document}
	
\title{A strongly convergent Krasnosel'ski\v{\i}-Mann-type algorithm for finding a common fixed point of a countably infinite family of nonexpansive operators in Hilbert spaces} 
\author{Radu Ioan Bo\c{t}\thanks{University of Vienna, Faculty of Mathematics, Oskar-Morgenstern-Platz 1, A-1090 Vienna, Austria, email: radu.bot@univie.ac.at. Research partially supported by the Austrian Science Fund (FWF), project I 2419-N32.} \and Dennis Meier\thanks{University of Vienna, Faculty of Mathematics, Oskar-Morgenstern-Platz 1, A-1090 Vienna, Austria, email: meierd61@univie.ac.at. Research supported by the Austrian Science Fund (FWF), project I 2419-N32.}}
\date{\today}
\maketitle

\noindent \textbf{Abstract.} In this article, we propose a Krasnosel'ski\v{\i}-Mann-type algorithm for finding a common fixed point of a countably infinite family of nonexpansive operators $(T_n)_{n \geq 0}$ in Hilbert spaces. 
We formulate an asymptotic property which the family $(T_n)_{n \geq 0}$ has to fulfill such that the sequence generated by the algorithm converges strongly to the element in $\bigcap_{n \geq 0} \operatorname{Fix} T_n$ with minimum norm. Based on this, we derive a 
forward-backward algorithm that allows variable step sizes and generates a sequence of iterates that converge strongly to the zero with minimum norm of the sum of a maximally monotone operator and a cocoercive one. We demonstrate the superiority of the forward-backward algorithm with 
variable step sizes over the one with constant step size by means of numerical experiments on variational image reconstruction and split feasibility problems in infinite dimensional Hilbert spaces.


	
\section{Introduction}	

Let $\mathcal{H}$ be a real Hilbert space with inner product $\langle \cdot , \cdot \rangle$ and induced norm $\|\cdot\| := \sqrt{\langle \cdot , \cdot \rangle}$. For a given nonexpansive (i.e. $1$-Lipschitz continuous) mapping $T : \mathcal{H} \to \mathcal{H}$, one of the most prominent iterative methods for finding a fixed point of $T$ is the so-called \textit{Krasnosel'ski\v{\i}-Mann algorithm}, which reads as (see \cite{BC})
\begin{align}\label{1}
x_{n+1} = x_n + \lambda_n (T x_n - x_n) ~~ \forall n \geq 0,
\end{align}
where $x_0 \in \mathcal{H}$ is an arbitrary starting point and $(\lambda_n)_{n \geq 0}$ is a sequence of nonnegative real numbers. Under the assumption that the set of fixed points of $T$ is nonempty, one can show under mild conditions imposed on $(\lambda_n)_{n \geq 0}$ that the sequence $(x_n)_{n \geq 0}$ converges weakly to a fixed point of $T$. 

Since the solving of many monotone inclusion and convex optimization problems can be reduced to the solving of a fixed point problem, there is a huge interest in designing corresponding efficient and stable algorithms. For instance, the \textit{forward-backward algorithm}, for determining a zero of the sum of a set-valued maximally monotone operator and a single-valued and cocoercive one, and the \textit{Douglas-Rachford algorithm}, for determining a zero of the sum of two set-valued maximally monotone operators, can be embedded in the framework of the Krasnosel'ski\v{\i}-Mann algorithm. A shortcoming which all of the above mentioned algorithms share is that the convergence of the generated iterates take place only with respect to the weak topology. In order to achieve strong convergence, one has to assume that the involved operators satisfy a stronger notion of monotonicity, like strong monotonicity (\cite{BC}). 

Since, however, many interesting problems to be solved are formulated in  infinite dimensional function spaces, where weak and strong convergence do not coincide, and do not involve strong monotone operators, there is a strong interest in developing algorithms, which generate iterates that strongly converge under minimal assumptions. A variant of the Krasnosel'ski\v{\i}-Mann algorithm which overcomes the drawback of weak convergence has been proposed in \cite{BoCsMe} and reads as follows:
\begin{align}\label{2}
x_{n+1} = \beta_n x_n + \lambda_n (T (\beta_n x_n) - \beta_n x_n) ~~ \forall n \geq 0,
\end{align} 
where $x_0 \in \mathcal{H}$ is an arbitrary starting point and $(\lambda_n)_{n \geq 0}$ and $(\beta_n)_{n \geq 0}$ are suitably chosen sequences of positive numbers. The formulation of the iterative scheme has its roots in a standard Tikhonov regularization approach. The sequence $(\beta_n)_{n \geq 0}$ is called \textit{Tikhonov regularization sequence} and has the role to enforce the strong convergence of $(x_n)_{n \geq 0}$ to the fixed point of $T$ with minimum norm. The iterative scheme \eqref{2} was the starting point for deriving in \cite{BoCsMe} strongly convergent forward-backward, Douglas-Rachford as well as primal-dual algorithms for monotone inclusion problems. Tikhonov regularization techniques has been used also in \cite{At, LM, X} in order to enforce strong convergence of numerical algorithms.

In this article we will address the problem of finding a common fixed point of a family of nonexpansive operators $T_n : \mathcal{H} \to \mathcal{H}$, for $n \geq 0$. We will formulate a  Krasnosel'ski\v{\i}-Mann-type algorithm endowed with Tikhonov regularization terms, which evaluates in the iteration $n \geq 0$ the operator $T_n$. Under the hypothesis that the sequence $(T_n)_{n \geq 0}$ fulfils an appropriate asymptotic condition, and provided that the intersection of the sets of fixed points of $(T_n)_{n \geq 0}$ is nonempty, we show that the sequence of generated iterates converges strongly to the common fixed point with minimum norm of the mappings $(T_n)_{n \geq 0}$. Based on this, we derive a strongly convergent forward-backward method with variable step sizes for finding the zeros of the sum of a maximally monotone operator and a cocoercive one. This method is much more flexible than the variant with constant step size, a fact which we also emphasize by means of numerical experiments on variational image reconstruction and split feasibility problems in infinite dimensional Hilbert spaces.

\section{A strongly convergent Krasnosel'ski\v{\i}-Mann-type algorithm}\label{sec2}

The symbols $\rightharpoonup$ and $\rightarrow$ denote weak and strong convergence, respectively. We recall that a mapping $T : \mathcal{H} \to \mathcal{H}$ is called nonexpansive, if $\|Tx - Ty\| \leq \|x-y\|$ for all $x,y \in \mathcal{H}$. The set of fixed points of $T$ is denoted by
$\operatorname{Fix} T:=\{x \in \mathcal{H}: T(x)=x\}$. For a nonempty convex and closed set $C \subseteq \mathcal{H}$, the projection operator onto $C$, $\operatorname{P}_C : \mathcal{H} \to \mathcal{H}$, is defined as $\operatorname{P}_C(x) := \operatorname{argmin}\limits_{c \in C} \| x - c \|$.

The following result, which is a direct consequence of \cite[Lemma~2.5]{X}, will play a key role in the proof of the main result of this paper, which we formulate and prove subsequently.
	
\begin{lemma}\label{lemma 4}
	Let $(a_n)_{n \geq 0}$ be a sequence of non-negative real numbers satisfying the inequality 
	\[a_{n+1} \leq (1 - \theta_n)a_n + \theta_n b_n + \epsilon_n \phantom{xx} \forall n \geq 0, \]
	where \\
	(i) $0 \leq \theta_n \leq 1$ for all $n \geq 0$ and $\sum_{n \geq 0} \theta_n = + \infty$;\\
	(ii) $\limsup_{n \to + \infty} b_n \leq 0$;\\
	(iii) $\epsilon_n \geq 0$ for all $n \geq 0$ and $\sum_{n \geq 0} \epsilon_n < + \infty$. \\[5pt]
	Then the sequence $(a_n)_{n \geq 0}$ converges to $0$. 
\end{lemma}	
	
\begin{theorem}\label{thm1}
	Let $(T_n)_{n \geq 0}$ be a sequence of nonexpansive operators $T_n : \mathcal{H} \to \mathcal{H}$, for $n \geq 0$, with $S := \bigcap_{n \geq 0} \operatorname{Fix} T_n \neq \emptyset$.
	Consider the iterative scheme 
	\begin{align}\label{eq1}
	x_{n+1} = \beta_n x_n + \lambda_n (T_n (\beta_n x_n) - \beta_n x_n) \phantom{xx} \forall n \geq 0,
	\end{align}
	with starting point $x_0 \in \mathcal{H}$ and $(\lambda_n)_{n \geq 0}$, $(\beta_n)_{n \geq 0}$ real sequences satisfying the conditions: \\[5pt]
	(i) $0 < \beta_n \leq 1$ for any $n \geq 0$, $\lim_{n \to +\infty} \beta_n = 1$, $\sum_{n \geq 0} (1 - \beta_n) = + \infty$ and $\sum_{n \geq 1} |\beta_n - \beta_{n-1}| < + \infty$; \\
	(ii) $0 < \lambda_n \leq 1$ for any $n \geq 0$, $\liminf_{n \to +\infty} \lambda_n > 0$ and $\sum_{n \geq 1} |\lambda_n - \lambda_{n-1}| < + \infty$; \\
	(iii)
	\begin{align*}
	\sum_{n \geq 1} \|T_n(\beta_{n-1} x_{n-1}) - T_{n-1}(\beta_{n-1} x_{n-1})\| < + \infty.
	\end{align*} 
	Then 
	\begin{align*}
	\|x_n - T_n x_n \| \to 0 \text{ as } n \to + \infty.
	\end{align*}
	In addition, if we suppose that the sequence $(T_n)_{n \geq 0}$ satisfies the asymptotic condition
	\begin{align}\label{asymptotic}
	&\text{for any subsequences } (T_{n_k})_{k \geq 0} \text{ of } (T_n)_{n \geq 0} \text{ and } (x_{n_k})_{k \geq 0} \text{ of } (x_n)_{n \geq 0} \nonumber \\ &\text{with } x_{n_k} \rightharpoonup x \in \mathcal{H} \text{ and } x_{n_k} - T_{n_k} x_{n_k} \to 0 \text{ as } k \to + \infty \text{ it holds } x \in \bigcap_{n \geq 0} \operatorname{Fix} T_n, 
	\end{align} 
	then $(x_n)_{n \geq 0}$ converges strongly to $\operatorname{P}_S (0)$. 
\end{theorem}	
	
\begin{proof}
First, we show that the sequence $(x_n)_{n \geq 0}$ is bounded. Let be $x \in S$.
Due to the nonexpansiveness of $T_n$, we have for any $n \geq 0$
\begin{align*}
\|x_{n + 1} - x \| &= \|(1 -  \lambda_n) (\beta_n x_n - x) +  \lambda_n (T_n(\beta_n x_n) - T_n x) \| \\ 
&\leq (1 -  \lambda_n) \| (\beta_n x_n - x) \| +  \lambda_n \| (T_n(\beta_n x_n) - T_n x) \| \\
&\leq \| \beta_n x_n - x \| = \| \beta_n (x_n - x) + (\beta_n - 1)x \| \\
&\leq \beta_n \| x_n - x \| + (1 - \beta_n) \| x \|.
\end{align*}
From here it follows that
\begin{align*}
\| x_n - x \| \leq \operatorname{max}\{\| x_0 - x \|, \| x \|\} \phantom{xx} \forall n \geq 0,
\end{align*}
what shows that $(x_n)_{n \geq 0}$ is bounded. Since $T_n$ is nonexpansive, we also have for any $n \geq 0$
\begin{align*}
\|T_n(\beta_n x_n) - x \| = \|T_n(\beta_n x_n) - T_n x \| \leq \|\beta_n x_n - x \|,
\end{align*}
This shows that the sequence $(T_n(\beta_n x_n))_{n \geq 0}$ is also bounded. 

Next, we prove that
\begin{align}\label{eq3}
\| x_{n+1} - x_n \| \to 0 \text{ as } n \to + \infty.
\end{align}
Indeed, by taking into account that the operators $(T_n)_{n \geq 0}$ are nonexpansive, we can derive for any $n \geq 1$ the following estimate
\begin{align*}
& \ \| x_{n+1} - x_n \|  \\ 
= & \ \| (1 -  \lambda_n) \beta_n x_n - (1 -  \lambda_{n-1}) \beta_{n-1} x_{n-1} +  \lambda_n T_n(\beta_n x_n) -  \lambda_{n-1} T_{n-1}(\beta_{n-1} x_{n-1}) \| \\[5pt]
\leq & \ \| (1 -  \lambda_n) (\beta_n x_n - \beta_{n-1} x_{n-1}) + ( \lambda_{n-1} -  \lambda_n) \beta_{n-1} x_{n-1} \| \\
& \  + \|  \lambda_n (T_n(\beta_n x_n) - T_{n-1}(\beta_{n-1} x_{n-1}) ) +  \lambda_n T_{n-1}(\beta_{n-1} x_{n-1}) -  \lambda_{n-1} T_{n-1}(\beta_{n-1} x_{n-1}) \| \\[5pt]
\leq & \ \| \beta_n x_n - \beta_{n-1} x_{n-1} \| + | \lambda_{n-1} -  \lambda_n| \| \beta_{n-1} x_{n-1} \| \\
& \ + \|  \lambda_n (T_n(\beta_{n-1} x_{n-1}) - T_{n-1}(\beta_{n-1} x_{n-1})) + ( \lambda_{n} -  \lambda_{n-1}) T_{n-1}(\beta_{n-1} x_{n-1}) \| \\[5pt]
\leq & \ \| \beta_n x_n - \beta_{n-1} x_{n-1} \| + | \lambda_{n-1} -  \lambda_n| \big(\| \beta_{n-1} x_{n-1} \| + \| T_{n-1}(\beta_{n-1}x_{n-1})\| \big) \\
& \ + \lambda_n \| T_n(\beta_{n-1} x_{n-1}) - T_{n-1}(\beta_{n-1} x_{n-1})\|.
\end{align*}
Thanks to the boundedness of the sequences $(\beta_n x_n)_{n \geq 0}$ and $(T_n(\beta_n x_n))_{n \geq 0}$, there exists a constant $C_2 > 0$ such that
\begin{align*}
\| x_{n+1} - x_n \| \leq & \ \| \beta_n x_n - \beta_{n-1} x_{n-1} \| + | \lambda_{n-1} -  \lambda_n| C_2\\
& \ + \lambda_n \| T_n(\beta_{n-1} x_{n-1}) - T_{n-1}(\beta_{n-1} x_{n-1})\| \ \forall n \geq 1.
\end{align*}
Since $(x_n)_{n \geq 0}$ is bounded, there exists a constant $C_1 >0$ such that
\begin{align*}
\| x_{n+1} - x_n \| \leq  & \ \| \beta_n (x_n - x_{n-1}) + (\beta_n - \beta_{n-1}) x_{n-1} \| \\ 
& + | \lambda_{n-1} -  \lambda_n| C_2 + \lambda_n \| T_n(\beta_{n-1} x_{n-1}) - T_{n-1}(\beta_{n-1} x_{n-1})\| \\
\leq & \ \beta_n \|x_n - x_{n-1}\| + |\beta_n - \beta_{n-1}| C_1 \\ 
&  + | \lambda_{n-1} -  \lambda_n| C_2 + \lambda_n \| T_n(\beta_{n-1} x_{n-1}) - T_{n-1}(\beta_{n-1} x_{n-1})\|.
\end{align*}
Now statement \eqref{eq3} is a consequence of Lemma \ref{lemma 4}, by taking for $n \geq 1$ the choices $a_n := \| x_n - x_{n-1} \|$, $b_n := 0$, $\theta_n := 1 - \beta_n$ and 
$$\epsilon_n := |\beta_n - \beta_{n-1}| C_1 + | \lambda_{n-1} -  \lambda_n| C_2 + \lambda_n \| T_n(\beta_{n-1} x_{n-1}) - T_{n-1}(\beta_{n-1} x_{n-1})\|.$$

Next we prove that 
\begin{align}\label{eq4}
\|x_n - T_n x_n \| \to 0 \text{ as } n \to + \infty.
\end{align}
For any $n \geq 0$ we have the following estimate
\begin{align*}
\| x_n - T_n x_n \| &\leq \| x_{n+1} - x_n \| + \| x_{n+1} - T_n x_n \| \\
&= \| x_{n+1} - x_n \| + \| (1 -  \lambda_n) (\beta_n x_n - T_n x_n ) +  \lambda_n (T_n(\beta_n x_n) - T_n x_n )\| \\
&\leq \| x_{n+1} - x_n \| + (1 -  \lambda_n) \| \beta_n x_n - T_n x_n \| +  \lambda_n \| \beta_n x_n -  x_n \| \\
&\leq \| x_{n+1} - x_n \| + (1 -  \lambda_n) \| \beta_n x_n - \beta_n T_n x_n  \| \\
& \phantom{xx} + (1 -  \lambda_n) \| \beta_n T_n x_n - T_n x_n \| +  \lambda_n (1 - \beta_n) \| x_n \| \\
&= \| x_{n+1} - x_n \| + (1 -  \lambda_n) \| x_n - T_n x_n  \| \\
& \phantom{xx} + (1 -  \lambda_n)(1 - \beta_n) \| T_n x_n \| +  \lambda_n (1 - \beta_n) \| x_n \|.
\end{align*}
From here it follows that
\begin{align*}
\lambda_n \| x_n - T_n x_n \| \leq \| x_{n+1} - x_n \| + (1 -  \lambda_n)(1 - \beta_n) \| T_n x_n \| +  \lambda_n (1 - \beta_n) \| x_n \| \ \forall n \geq 0.
\end{align*}
Since 
\begin{align*}
\| T_n x_n - x \| = \| T_n x_n - T_n x \| \leq \| x_n - x \| \leq \operatorname{max}\{ \| x_0 - x \|, \| x \| \},
\end{align*}
it follows that the sequence $(T_n x_n)_{n \geq 0}$ is also bounded. 
Taking into account that $(x_n)_{n \geq 0}$ is bounded, \eqref{eq3} and the assumptions (i) and (ii), it follows from the last inequality that \eqref{eq4} holds.

In order to prove the last statement of the theorem, we suppose that the sequence $(T_n)_{n \geq 0}$ fulfills the asymptotic condition \eqref{asymptotic}.
We denote by $\bar{x} := \operatorname{P}_S (0)$ the element in $S$ with minimum norm. Using again that $T_n$ is nonexpansive, we have for any $n \geq 0$
\begin{align*}
\|x_{n + 1} - \bar{x} \| &= \|(1 -  \lambda_n) (\beta_n x_n - \bar{x}) +  \lambda_n (T_n(\beta_n x_n) - T_n \bar{x}) \| \\ 
& \leq (1 -  \lambda_n) \| \beta_n x_n - \bar{x} \| +  \lambda_n \| T_n(\beta_n x_n) - T_n \bar{x} \| \\
& \leq \| \beta_n x_n - \bar{x} \|.
\end{align*}
Hence, 
\begin{align}\label{eq5}
\|x_{n + 1} - \bar{x} \|^2 &\leq \| \beta_n x_n - \bar{x} \|^2 = \| \beta_n (x_n - \bar{x}) + (\beta_n - 1) \bar{x} \|^2 \nonumber \\
&= \beta_n^2 \| x_n - \bar{x} \|^2 + 2 \beta_n (1 - \beta_n) \langle - \bar{x}, x_n - \bar{x} \rangle + (1 - \beta_n)^2 \| \bar{x} \|^2 \nonumber \\
&\leq \beta_n \| x_n - \bar{x} \|^2 + (1 - \beta_n) (2 \beta_n \langle - \bar{x}, x_n - \bar{x} \rangle + (1 - \beta_n) \| \bar{x} \|^2) \ \forall n \geq 0. 
\end{align}
Next we show that 
\begin{align}\label{eq6}
\limsup\limits_{n \to + \infty} \langle -\bar{x}, x_n - \bar{x} \rangle \leq 0.  
\end{align}
Assuming the contrary, there would exist a positive real number $l$ and a subsequence $(x_{n_k})_{k \geq 0}$ such that \[\langle -\bar{x}, x_{n_k} - \bar{x} \rangle \geq l > 0 \quad  \forall k \geq 0. \]
Due to the boundedness of the sequence $(x_n)_{n \geq 0}$, we can assume without losing the generality that $(x_{n_k})_{k \geq 0}$ weakly converges to an element $y \in \mathcal{H}$. Taking into account \eqref{eq4}, it follows from the asymptotic condition \eqref{asymptotic} that $y$ lies in $S$.
On the other hand, from the variational characterization of the projection we have
\begin{align*}
l \leq \lim\limits_{k \to + \infty} \langle - \bar{x}, x_{n_k} - \bar{x} \rangle = \langle - \bar{x}, y - \bar{x} \rangle \leq 0,
\end{align*}
which leads to a contradiction. This shows that \eqref{eq6} holds. Thus 
\begin{align*}
\limsup\limits_{n \to + \infty} (2 \beta_n \langle - \bar{x}, x_n - \bar{x} \rangle + (1 - \beta_n) \| \bar{x} \|^2) \leq 0. 
\end{align*}
A direct application of Lemma \ref{lemma 4} to \eqref{eq5}, by taking for $n \geq 0$ the choices $a_n := \| x_n - \bar{x} \|^2$, 
$$b_n := 2 \beta_n \langle - \bar{x}, x_n - \bar{x} \rangle + (1 - \beta_n) \| \bar{x} \|^2,$$ 
$\epsilon_n := 0$ and $\theta_n := 1 - \beta_n$, delivers the desired conclusion.
\end{proof}

\begin{remark} \upshape
(i) The asymptotic condition \eqref{asymptotic} was introduced in \cite{M} and used in the convergence analysis of inertial Krasnosel'ski\v{\i}-Mann algorithms designed for finding a common fixed point of a countably infinite family of nonexpansive operators.

(ii) In the particular case when $T_n = T$ for any $n \geq 0$, where $T : \mathcal{H} \to \mathcal{H}$ is a nonexpansive operator, the asymptotic condition \eqref{asymptotic} becomes the so-called \emph{demiclosedness principle}, which is known to hold for any nonexpansive operator (see \cite[Corollary~4.18]{BC}). 
In this setting, Theorem \ref{thm1} reduces to Theorem 3 in \cite{BoCsMe}.

(iii) The assumptions imposed on the sequence $(\lambda_n)_{n \geq 0}$ are met by every monotonically increasing or decreasing (and hence convergent) sequence with a positive limit. For the Tikhonov regularization sequence $(\beta_n)_{n \geq 0}$ one can choose, for instance, 
$\beta_0 \in (0, \frac{1}{2})$ and $\beta_n = 1 - \frac{1}{1+n}$ for any $n \geq 1$.
\end{remark}

The following result is a consequence of Theorem \ref{thm1} and addresses the problem of finding a common fixed point of a countably infinite family of averaged operators. It will play a determinant role in the convergence analysis of the forward-backward method with variable step sizes which we propose in the next section. 
We recall that a mapping $T : \mathcal{H} \to \mathcal{H}$ is called firmly nonexpansive, if $\|T x - T y \|^2 + \| (\operatorname{Id} - T) x - (\operatorname{Id} - T) y \|^2 \leq \|x - y\|^2$ for all $x,y \in \mathcal{H}$. Every firmly nonexpansive mapping is also nonexpansive. Let $\alpha \in (0,1)$ be fixed. 
We say that $R : \mathcal{H} \to \mathcal{H}$ is an $\alpha$-averaged operator, if there exists a nonexpansive operator $T : \mathcal{H} \to \mathcal{H}$ such that $R = (1-\alpha) \operatorname{Id} + \alpha T$. Here, $\operatorname{Id}$ denotes the identity operator on $\mathcal{H}$. 

\begin{corollary}\label{thm2}
        Let $(R_n)_{n \geq 0}$ be a sequence of $\alpha_n$-averaged operators $R_n : \mathcal{H} \to \mathcal{H}$ with $S := \bigcap_{n \geq 0} \operatorname{Fix} R_n \neq \emptyset$ and $(\alpha_n)_{n \geq 0} \subseteq (0,1)$ fulfilling the conditions 
\begin{align}\label{alpha} 
 \liminf_{n \to +\infty} \alpha_n > 0 \ \mbox{and} \ \sum_{n \geq 1} \left|\frac{1}{\alpha_n} - \frac{1}{\alpha_{n-1}} \right| < + \infty.
\end{align}
Consider the iterative scheme 
	\begin{align}\label{eq11}
	x_{n+1} = \beta_n x_n + \lambda_n (R_n (\beta_n x_n) - \beta_n x_n) \phantom{xx} \forall n \geq 0,
	\end{align}
	with starting point $x_0 \in \mathcal{H}$ and $(\lambda_n)_{n \geq 0}$, $(\beta_n)_{n \geq 0}$ real sequences satisfying the conditions:
	(i) $0 < \beta_n \leq 1$ for any $n \geq 0$, $\lim_{n \to +\infty} \beta_n = 1$, $\sum_{n \geq 0} (1 - \beta_n) = + \infty$ and $\sum_{n \geq 1} |\beta_n - \beta_{n-1}| < + \infty$; \\
	(ii) $0 < \lambda_n \leq \frac{1}{\alpha_n}$ for any $n \geq 0$, $\liminf_{n \to +\infty} \lambda_n > 0$ and $\sum_{n \geq 1} |\lambda_n - \lambda_{n-1}| < + \infty$; \\
	(iii) 
	\begin{align*}
	\sum_{n \geq 1} \|R_n(\beta_{n-1} x_{n-1}) - R_{n-1}(\beta_{n-1} x_{n-1})\| < + \infty.
	\end{align*} 
	Then
	\begin{align*}
	\|x_n - R_n x_n \| \to 0 \text{ as } n \to + \infty.
	\end{align*}
	In addition, if we suppose that the sequence $(R_n)_{n \geq 0}$ satisfies the asymptotic condition
	\begin{align}\label{asymptoticR}
	&\text{for any subsequences } (R_{n_k})_{k \geq 0} \text{ of } (R_n)_{n \geq 0} \text{ and } (x_{n_k})_{k \geq 0} \text{ of } (x_n)_{n \geq 0}, \nonumber \\ & \text{with } x_{n_k} \rightharpoonup x \in \mathcal{H} \text{ and } x_{n_k} - R_{n_k} x_{n_k} \to 0 \text{ as } k \rightarrow +\infty \text{ it holds } 
	x \in \bigcap_{n \geq 0} \operatorname{Fix} R_n, 
	\end{align} 
	then $(x_n)_{n \geq 0}$ converges strongly to $\operatorname{P}_S (0)$. 
\end{corollary}

\begin{proof}
	Fix $n \geq 0$. Let $T_n : \mathcal{H} \to \mathcal{H}$ be the nonexpansive operator such that $R_n = (1 - \alpha_n) \operatorname{Id} + \alpha_n T_n$. The iterative scheme in \eqref{eq11} can be rewritten as 
	\[x_{n+1} = \beta_n x_n + \alpha_n \lambda_n (T_n(\beta_n x_n) - \beta_n x_n) \ \forall n \geq 0.\]
We have $T_n = \frac{1}{\alpha_n} R_n + \left(1 - \frac{1}{\alpha_n}\right) \operatorname{Id}$ and therefore 
	\begin{align*}
	& \|T_n (\beta_{n-1} x_{n-1}) - T_{n-1} (\beta_{n-1} x_{n-1}) \| \\ 
	\leq & \ \frac{1}{\alpha_n} \|R_n(\beta_{n-1} x_{n-1}) - R_{n-1} (\beta_{n-1} x_{n-1})\| + \left |\frac{1}{\alpha_n}- \frac{1}{\alpha_{n-1}} \right| \|R_{n-1} (\beta_{n-1} x_{n-1})\| \\ 
	& \ + \left |\frac{1}{\alpha_n}- \frac{1}{\alpha_{n-1}} \right| \|\beta_{n-1} x_{n-1}\|.
	\end{align*}
	From the proof of Theorem \ref{thm1} it follows that the sequences $(R_n(\beta_{n} x_{n}))_{n \geq 0}$ and $(\beta_n x_n)_{n \geq 0}$ are bounded. Combining this with \eqref{alpha} and assumption (iii), it follows that
	$$\sum_{n \geq 1} \|T_n(\beta_{n-1} x_{n-1}) - T_{n-1}(\beta_{n-1} x_{n-1})\| < + \infty.$$
Finally, since $\operatorname{Fix} T_n = \operatorname{Fix} R_n$, 
\begin{align*}
	\|x_{n} - T_{n} x_{n}\| = \frac{1}{\alpha_n} \|x_{n} - R_{n} x_{n}\|,
	\end{align*}
and $\liminf_{n \to +\infty} \alpha_n > 0$, it is easy to see that, if $(R_n)_{n \geq 0}$ satisfies the asymptotic condition \eqref{asymptoticR}, then $(T_n)_{n \geq 0}$ satisfies the asymptotic condition \eqref{asymptotic}. The conclusion follows by applying Theorem \ref{thm1}. 
\end{proof}

\section{A strongly convergent forward-backward algorithm with variable step sizes}\label{sec3}

Based on the general scheme proposed in Theorem \ref{thm1}, we will formulate in this section a strongly convergent forward-backward algorithm with variable step sizes for solving the monotone inclusion problem
\begin{align}\label{mon}
\mbox{find} \ x \in \mathcal{H} \ \mbox{such that} \ 0 \in Ax + Bx,
\end{align}
where $A : \mathcal{H} \rightrightarrows \mathcal{H}$ is a maximally monotone operator and $B : \mathcal{H} \to \mathcal{H}$ is a cocoercive operator.

Having a set-valued operator $A : \mathcal{H} \rightrightarrows \mathcal{H}$, we denote by $\operatorname{zer} A := \{x \in \mathcal{H} : 0 \in Ax \}$ its set of zeros, by $\operatorname{Gr}A := \{(x,u) \in \mathcal{H} \times \mathcal{H} : u \in Ax \}$ its graph and by $A^{-1} : \mathcal{H} \rightrightarrows \mathcal{H}$ the inverse operator of $A$, which is the operator having as graph $\operatorname{Gr} A^{-1} := \{(x,u) \in \mathcal{H} \times \mathcal{H} : x \in Au \}$. The operator $A$ is said to be monotone, if $\langle x - y, u - v \rangle \geq 0$ for all $(x,u),(y,v) \in \operatorname{Gr}A$, and maximally monotone, if it is monotone and there exists no proper monotone extension of the graph of $A$ on $\mathcal{H} \times \mathcal{H}$. The resolvent of $A$, $J_A : \mathcal{H} \rightarrow \mathcal{H}$, is defined by 
\[ J_A := (\operatorname{Id} + A)^{-1}.\]
If $A$ is maximally monotone, then $J_A : \mathcal{H} \rightarrow \mathcal{H}$ is single-valued, maximally monotone and firmly nonexpansive (see \cite[Proposition~23.7~and~Corollary~23.10]{BC}). Further, we denote by $R_A : \mathcal{H} \to \mathcal{H}, R_A := 2J_A - \operatorname{Id}$, the reflected resolvent of $A$.

We recall that, if $f : \mathcal{H} \to \mathbb{R} \cup \{+ \infty \}$ is a proper, convex and lower semicontinuous function, then the (convex) subdifferential $\partial f : \mathcal{H} \rightrightarrows \mathcal{H}$ of $f$, 
\[\partial f (x) := \{ p \in \mathcal{H} : f(y) - f(x) \geq \langle p, y - x \rangle \phantom{x} \forall y \in \mathcal{H}\}, \] for $x \in \mathcal{H}$ with $f(x) \neq + \infty$ and as $\partial f(x) = \emptyset$, otherwise, is a maximally monotone operator (see \cite{R}). Its resolvent is given by $J_{\partial f} = \operatorname{prox}_f$ (see \cite{BC}), where 
\[\operatorname{prox}_f : \mathcal{H} \to \mathcal{H}, \ \operatorname{prox}_f(x) = \operatorname{argmin}\limits_{y \in \mathcal{H}} \left\{ f(y) + \frac{1}{2} \|y-x\|^2 \right\},\] denotes the proximal operator of $f$.
For a nonempty closed and convex set $C \subseteq \mathcal{H}$, one has that $\operatorname{P}_C = \operatorname{prox}_{\delta_C}$, where \[
\delta_C(x)=\left\{\begin{array}{ll} 0, & x\in C \\
+ \infty, & x\not\in C \end{array}\right. ,
\]
denotes the indicator function of $C$.

A single-valued operator $B : \mathcal{H} \to \mathcal{H}$ is called $\beta$-cocoercive, for $\beta > 0$, if $\langle x-y, Bx-By \rangle \geq \beta \|Bx-By\|^2$ for all $x,y \in \mathcal{H}$. According to the Baillon-Haddad Theorem, if $g: \mathcal{H} \rightarrow \mathbb{R}$ is a convex and Fr\'echet differentiable function with $\frac{1}{\beta}$-Lipschitz gradient, then $\nabla g  : \mathcal{H} \to \mathcal{H}$ is a $\beta$-cocoercive operator.

The strongly convergent forward-backward algorithm with variable step sizes for solving \eqref{mon} that we propose in this section will be formulated as a particular case of \eqref{eq11} for $R_n = J_{\gamma_n A}(\operatorname{Id} - \gamma_n B)$, for $n \geq 0$, where $\operatorname{inf}_{n \geq 0} \gamma_n > 0$. To this end we will prove that $(R_n)_{n \geq 0}$ fulfills the asymptotic condition \eqref{asymptoticR} (see \cite[Corollary~17]{D}). The proof relies on the following lemma, which is a special instance of \cite[Corollary~25.5]{BC}.

\begin{lemma}\label{zer}
	Let $A, B : \mathcal{H} \rightrightarrows \mathcal{H}$ be maximally monotone operators, and $(x_n, u_n)_{n \geq 0} \in \operatorname{Gr} A$, $(y_n, v_n)_{n \geq 0} \in \operatorname{Gr} B$ such that $x_n \rightharpoonup x$, $y_n \rightharpoonup y$, $u_n \rightharpoonup u$, $v_n \rightharpoonup v$, $u_n + v_n \to 0$, $x_n - y_n \to 0$ as $n \to + \infty$. Then $x = y \in \operatorname{zer}(A+B)$, $(x, u) \in \operatorname{Gr}A$ and $(y, v) \in \operatorname{Gr}B$.
\end{lemma}
	
\begin{proposition}\label{prop1}
Let $A : \mathcal{H} \rightrightarrows \mathcal{H}$ be maximally monotone, $B : \mathcal{H} \to \mathcal{H}$ be $\beta$-cocoercive, for $\beta >0$, and $\operatorname{inf}_{n \geq 0} \gamma_n > 0$. Suppose that $\operatorname{zer}(A+B) \neq \emptyset$, and set 
$$R_n := J_{\gamma_n A}(\operatorname{Id} - \gamma_n B) \ \forall n \geq 0.$$ 
Then $(R_n)_{n \geq 0}$ fulfils the asymptotic condition \eqref{asymptoticR}.
\end{proposition}

\begin{proof}
By \cite[Proposition~25.1(iv)]{BC} we have $\operatorname{Fix} R_n = \operatorname{zer}(A+B)$, hence $\bigcap_{n \geq 0} \operatorname{Fix} R_n = \operatorname{zer}(A+B)$. 

Let $(x_n)_{n \geq 0}$ be the sequence generated by the iterative scheme \eqref{eq11}. Further, let $(x_{n_k})_{k \geq 0}$ be a subsequence of $(x_n)_{n \geq 0}$ such that $x_{n_k} \rightharpoonup x \in \mathcal{H}$ and $x_{n_k} - R_{n_k} x_{n_k} \to 0$ as $k \to + \infty$. We set $y_k := R_{n_k} x_{n_k}$ for any $k \geq 0$. It holds $x_{n_k} - y_k \to 0$, therefore $y_k \rightharpoonup x$ as $k \to + \infty$. Since 
\begin{align*}
x_{n_k} - y_k \in (\operatorname{Id} + \gamma_{n_k} A) y_k + \gamma_{n_k} Bx_{n_k} - y_k = \gamma_{n_k} (Ay_k + Bx_{n_k}) \ \forall k \geq 0, 
\end{align*}
it follows that for any $k \geq 0$ there exist $ (x_{n_k}, u_k) \in \operatorname{Gr} B, (y_k, v_k) \in \operatorname{Gr} A$ such that
\begin{align*}
\gamma_{n_k} (v_k + u_k) = x_{n_k} - y_k \to 0 \ \mbox{as} \ k \rightarrow +\infty. 
\end{align*}
Since $\operatorname{inf}_{n \geq 0} \gamma_n > 0$, we obtain $v_k + u_k \to 0$ as $k \to + \infty$.

On the other hand, since $B$ is $\beta$-cocoercive, we have $\|B x_{n_k}\| \leq \beta^{-1} \|x_{n_k}\| + \|B 0\|$ for any $k \geq 0$. Further, as $(x_{n_k})_{k \geq 0}$ is weakly convergent, the uniform boundedness principle implies that $(x_{n_k})_{k \geq 0}$ is bounded, hence $(Bx_{n_k})_{k \geq 0}$ is also bounded. Consequently, there exists a convergent subsequence $(Bx_{n_{k_l}})_{l \geq 0}$ of $(Bx_{n_k})_{k \geq 0}$ such that $Bx_{n_{k_l}} = u_{k_l} \rightharpoonup u$ as $l \to + \infty$. Since $v_k + u_k \to 0$ as $k \to + \infty$ it follows that $v_{k_l} \rightharpoonup - u$ as $l \to + \infty$. 

Finally, Lemma \ref{zer} applied to the sequences $(x_{n_{k_l}})_{l \geq 0}$, $(y_{k_l})_{l \geq 0}$, $(u_{k_l})_{l \geq 0}$ and $(v_{k_l})_{l \geq 0}$ gives $x = y \in \operatorname{zer}(A+B)$. 
\end{proof}

The following lemma will be useful in the proof of the main result of this section.

\begin{lemma}\label{FB}
	Let $A : \mathcal{H} \rightrightarrows \mathcal{H}$ be a maximally monotone operator and $(\gamma_n)_{n \geq 0} \subseteq (0,+\infty)$. Then for any $x \in \mathcal{H}$ and any $n, m \geq 0$ we have 
	\begin{align*}
	\left\| J_{\gamma_n A} (x - \gamma_n Bx)  - J_{\gamma_{m} A} (x - \gamma_{m} Bx) \right\| \leq \left| 1 - \frac{\gamma_{m}}{\gamma_n} \right| \left\| J_{\gamma_n A} (x - \gamma_n Bx) - x \right\|.
	\end{align*}
\end{lemma}

\begin{proof}
Let be $x \in \mathcal{H}$ and $n, m \geq 0$. According to \cite[Proposition 23.28(i)]{BC}, we have  that 
\[J_{\gamma_n A}x = J_{\gamma_{m} A}\left( \frac{\gamma_{m}}{\gamma_n} x + \left(1-\frac{\gamma_{m}}{\gamma_n}\right) J_{\gamma_n A}x \right). \] 
Further, using the nonexpansivity of the resolvent $J_{\gamma_{m}A}$, we get
	\begin{align*}
	&\left\| J_{\gamma_n A} (x - \gamma_n Bx)  - J_{\gamma_{m} A} (x - \gamma_{m} Bx) \right\| \\
	\leq& \left\| \frac{\gamma_{m}}{\gamma_n} (x - \gamma_n Bx) + \left(1-\frac{\gamma_{m}}{\gamma_n}\right) J_{\gamma_n A}(x - \gamma_n Bx) - (x - \gamma_{m}Bx) \right\| \\
	=& \left\| \frac{\gamma_{m}}{\gamma_n}x + \left(1-\frac{\gamma_{m}}{\gamma_n}\right) J_{\gamma_n A}(x - \gamma_n Bx) - x \right\| \\
	=& \left| 1 - \frac{\gamma_{m}}{\gamma_n} \right| \left\| J_{\gamma_n A} (x - \gamma_n Bx) - x \right\|. \qedhere
	\end{align*}
\end{proof}

\begin{theorem}\label{thm3}
	Let $A : \mathcal{H} \rightrightarrows \mathcal{H}$ be a maximally monotone operator and \mbox{$B : \mathcal{H} \rightarrow \mathcal{H}$} a $\beta$-cocoercive operator, for $\beta > 0$, such that $\operatorname{zer}(A+B) \neq \emptyset$. Consider the iterative scheme 
	\begin{align}\label{eq111}
	x_{n+1} = (1 - \lambda_n) \beta_n x_n + \lambda_n J_{\gamma_n A} (\beta_n x_n - \gamma_n B(\beta_n x_n)) \phantom{xx} \forall n \geq 0,
	\end{align}
	with starting point $x_0 \in \mathcal{H}$ and $(\lambda_n)_{n \geq 0}$, $(\beta_n)_{n \geq 0}$ and $(\gamma_n)_{n \geq 0}$ real sequences satisfying the conditions:\\[3pt]
(i) $0 < \beta_n \leq 1$ for any $n \geq 0$, $\lim_{n \to +\infty} \beta_n = 1$, $\sum_{n \geq 0} (1 - \beta_n) = + \infty$ and $\sum_{n \geq 1} |\beta_n - \beta_{n-1}| < + \infty$. \\
	(ii) $0 < \lambda_n \leq \frac{4\beta - \gamma_n}{2\beta}$ for any $n \geq 0$, $\liminf_{n \to +\infty} \lambda_n > 0$ and $\sum_{n \geq 1} |\lambda_n - \lambda_{n-1}| < + \infty$. \\
	(iii) $0 < \gamma_n < 2\beta$ for any $n \geq 0$, $\liminf_{n \to +\infty} \gamma_n > 0$, and $\sum_{n \geq 1} \left|\gamma_n - \gamma_{n-1} \right| < + \infty$. \\[5pt]
	Then $(x_n)_{n \geq 0}$ converges strongly to $\operatorname{P}_{\operatorname{zer}(A+B)}(0)$.
\end{theorem}

\begin{proof}
	The iterative scheme in \eqref{eq111} can be rewritten as 
	\begin{align*}
	x_{n+1} = \beta_n x_n + \lambda_n (R_n (\beta_n x_n) - \beta_n x_n) \phantom{xx} \forall n \geq 0,
	\end{align*}
	for $R_n := J_{\gamma_n A} \circ (\operatorname{Id} - \gamma_n B)$. For any $n \geq 0$ we have that $\operatorname{Fix} R_n = \operatorname{zer}(A + B)$ (see \cite[Proposition~25.1]{BC}), thus 
$\bigcap_{n \geq 0} \operatorname{Fix} R_n = \operatorname{zer}(A + B) \neq \emptyset$.

We set $\alpha_n := \frac{2 \beta}{4 \beta - \gamma_n}$ for any $n \geq 0$. According to (iii) we have that $\liminf_{n \to + \infty} \alpha_n > 0$ and
	\begin{align*}
	\sum_{n \geq 1}\left|\frac{1}{\alpha_n} - \frac{1}{\alpha_{n-1}} \right| = \sum_{n \geq 1}  \left|\frac{4\beta - \gamma_n}{2\beta} - \frac{4\beta - \gamma_{n-1}}{2\beta} \right| = \frac{1}{2 \beta} \sum_{n \geq 1}  |\gamma_n - \gamma_{n-1}| < +\infty,
	\end{align*}
hence \eqref{alpha} holds.

Thanks to Lemma \ref{FB}, we have that for any $n \geq 1$
	\begin{align*}
	\|R_n(\beta_{n-1} x_{n-1}) - R_{n-1}(\beta_{n-1} x_{n-1})\| \leq \left| 1 - \frac{\gamma_{n-1}}{\gamma_n} \right| \left\| R_n(\beta_{n-1} x_{n-1}) - \beta_{n-1} x_{n-1} \right\|.
	\end{align*}
We have seen in the proof of Corollary \ref{thm2} that $(\beta_n x_n)_{n \geq 0}$ and $(T_n(\beta_{n-1} x_{n-1}))_{n \geq 1}$ are bounded, too. By (iii) it follows that $\sum_{n \geq 0} \left| 1 - \frac{\gamma_{n-1}}{\gamma_n} \right| < + \infty$
	and we can conclude that 
$$\sum_{n \geq 1} \|R_n(\beta_{n-1} x_{n-1}) - R_{n-1}(\beta_{n-1} x_{n-1})\| < + \infty,$$
which means that the conditions (i)-(iii) in Corollary \ref{thm2} are fulfilled.

By Proposition \ref{prop1} it follows that the sequence $(R_n)_{n \geq 0}$ fulfils condition \eqref{asymptoticR}. All assumptions in Corollary \ref{thm2} are fulfilled, which leads to the desired conclusion.
\end{proof}

\begin{remark} \label{rem2} \upshape
   (i) Let $f : \mathcal{H} \to (-\infty, +\infty]$ be a proper, convex and lower-semicontinuous function 
	and $g : \mathcal{H} \to \mathbb{R}$ a convex and Fr\'echet differentiable function with $\frac{1}{\beta}$-Lipschitz continuous gradient, for $\beta > 0$, such that $\operatorname{argmin}(f+g) \neq \emptyset$. The iterative scheme 
\begin{align}\label{itstep}
	x_{n+1} = (1 - \lambda_n) \beta_n x_n + \lambda_n \operatorname{prox}_{\gamma_n f} (\beta_n x_n - \gamma_n \nabla g(\beta_n x_n)) \phantom{xx} \forall n \geq 0,
	\end{align}
	with starting point $x_0 \in \mathcal{H}$ and $(\lambda_n)_{n \geq 0}$, $(\beta_n)_{n \geq 0}$ and $(\gamma_n)_{n \geq 0}$ real sequences satisfying the conditions (i)-(iii) in Theorem \ref{thm3} generates a sequence $(x_n)_{n \geq 0}$ which converges strongly to ${\operatorname P}_{\operatorname{argmin}(f+g) }(0)$.

(ii) Based on Theorem \ref{thm3}, one can derive strongly convergent primal-dual splitting algorithms with variable step sizes (\cite{AtBACo, BoCsHein, BoHe, BaCo, CoPe, V}) of forward-backward type. These methods are known for their high efficency when solving highly structured monotone inclusion problems involving mixtures of linearly composed maximally monotone operators and parallel sums of maximally monotone operators. The derivation of strongly convergent primal-dual splitting methods with variable step sizes of forward-backward type can be done in an analogous way as in \cite[Section~5.1]{BoCsMe}. 

   (iii) For $A, B : \mathcal{H} \rightrightarrows \mathcal{H}$ two maximally monotone operators, the classical Douglas-Rachford algorithm operates according to the iterative scheme 
   \begin{align*}
   (\forall n \geq 0)~\left\{\begin{array}{ll} y_n = J_{\gamma B} x_n \\
   z_n = J_{\gamma A}(2y_n - x_n) \\ 
   x_{n+1} = x_n + \lambda_n (z_n - y_n) \end{array}\right. ,
   \end{align*}
with starting point $x_0 \in \mathcal{H}$. Under mild conditions imposed on the sequence $(\lambda_n)_{n \geq 0}$ and under the assumption that $\operatorname{zer}(A+B) \neq \emptyset$, there exists an element $x \in \operatorname{Fix} R_{\gamma A} R_{\gamma B}$ such that the sequence $(x_n)_{n \geq 0}$ converges weakly to $x$ and $(y_n)_{n \geq 0}$, and $(z_n)_{n \geq 0}$ converge weakly to $J_{\gamma B}$ (see \cite[Theorem~25.6]{BC}). One can design a strongly convergent Douglas-Rachford algorithm with variable step sizes from the setting of Theorem \ref{thm1}, by considering $T_n : \mathcal{H} \to \mathcal{H}$, $T_n := R_{\gamma_n A} \circ R_{\gamma_n B}$ for $n \geq 0$. However, it is not clear if the family of operators $(T_n)_{n \geq 0}$ satisfies the asymptotical condition \eqref{asymptotic}. On the other hand, for a constant step size, i.e. $\gamma_n = \gamma$ for any $n \geq 0$ and some $\gamma \in (0,2\beta)$, a strongly convergent Douglas-Rachford method can be found in \cite{BoCsMe}.   
\end{remark}

\section{Numerical experiments}\label{sec4}

The numerical experiments presented in this section were implemented in Mathematica on a 4 $\times$ Intel\textregistered~ Core\texttrademark ~ i5-4670S CPU~@~3.10 GHz computer with 8GB of RAM.
It is worth to mention that the implementations were carried out by only using the symbolic computation packages of Mathematica. In other words, no discretization was performed. By doing so, we could fully exploit the spirit of the Tikhonov regularization technique which is employed in the algorithms we propose in this paper.

\subsection{A variational minimization problem}\label{subsec41}

For the first numerical experiment we considered an optimization problem which occurs in variational image reconstruction in infinite dimensional Hilbert spaces. For more details concerning this topic, we refer the reader to \cite{Ch} and the references therein. Let $\Omega \subset \mathbb{R}^n$, $n \geq 1$, be a domain, $\mathcal{H} := L^2(\Omega) := \left\{ u : \Omega \to \mathbb{R} : \int_\Omega |u(x)|^2 dx < + \infty \right\}$ be equipped with the scalar product $\langle u, v \rangle := \int_\Omega u (x) v(x) dx$ and the associated norm $\|u\| := \left( \int_\Omega |u(x)|^2 dx \right)^{1/2}$ for all $u, v \in L^2(\Omega)$. The optimization problem  we solved has the following formulation
\begin{align}\label{rec}
	\min\limits_{u \in \mathcal{H}} \left\{\frac{\lambda}{2} \int_{\Omega} ((Ku)(x) - b(x))^2 dx + \frac{1}{2} \int_{\Omega} F(x, u(x)) dx \right\}, 
\end{align}
where $\lambda > 0$, $b \in L^2(\Omega)$,
\begin{align*}
K : \mathcal{H} \to \mathcal{H}, \ (Ku)(x) := \int_\Omega k(x,y) u(y) dy, 
\end{align*} 
with $k \in L^2(\Omega \times \Omega)$, and  $F : \Omega \times \mathcal{H}  \to \mathbb{R}$ a measurable function. 
The parameter $\lambda$ controls the trade-off between a good fit of $u$ and a smoothness
requirement due to the regularization term, while the integral operator $K$ models the deblurring of the original image $u$. The linear operator $K$ is bounded with $\|K\| \leq \|k\|_{L^2(\Omega \times \Omega)}$ and its adjoint operator $K^\ast : \mathcal{H} \to \mathcal{H}$ is given by $(K^\ast u)(x) := \int_\Omega k(y,x) u(y) dy$ for any $ x \in \Omega$ (see, for instance, \cite[Chapter~1]{A}). \\[5pt]
For our numerical example, we considered $\Omega := (0,1)$, $F := u^2$ and 
\begin{align*}
k(x,y) := \left\{\begin{array}{ll} 1, & y \leq x \\
0, & y > x \end{array}\right. .
\end{align*}
For this choice, $K$ is the so-called \textit{Volterra operator} given by 
\begin{align*}
(Ku)(x) = \int_{0}^{x} u(y)~ dy.
\end{align*}
We have that
\begin{align*}
\|K\|^2 \leq \|k\|^2_{L^2((0,1) \times (0,1))} =\! \int_{0}^{1} \int_{0}^1 k(x,y)^2 dx~dy = \int_{0}^{1} \int_{y}^{1} 1~ dx~dy = \int_{0}^{1} (1 - y)~dy = \frac{1}{2},
\end{align*} 
while the adjoint operator of $K$ is given by 
\begin{align*}
K^\ast : \mathcal{H} \to \mathcal{H}, \ (K^\ast u)(x) = \int_{x}^1 u(y)~ dy. 
\end{align*}
In this setting, problem \eqref{rec} becomes 
\begin{align}\label{T}
	\min\limits_{u \in \mathcal{H}} \left\{\frac{\lambda}{2} \int_{\Omega} ((Ku)(x) - b(x))^2 dx + \frac{1}{2} \int_{(0,1)} u^2(x) dx,  \right\}, 
\end{align}
Since both operator summands are differentiable with Lipschitz continuous gradient, we had the options: (1) to use the iterative scheme \eqref{itstep} as a gradient method (i.e. $f$ is identical zero and $g$ is the function in the objective); (2) to divide the objective into two parts and evaluate one of the two smooth functions via its proximal operator, hence, end up with a proximal-gradient scheme. We pursued both approaches, by taking also into account \cite{CG}, which suggests that the evaluation of a smooth objective activated via its proximal operator may be advantageous in terms of computational performance compared to evaluating the whole objective through its gradient.
 
When choosing $f$ identical zero and $g(u)= \frac{\lambda}{2} \int_{\Omega} ((Ku)(x) - b(x))^2 dx + \frac{1}{2} \int_{(0,1)} u^2(x) dx$, the gradient of $g$ reads
\begin{align*}
\nabla g(u) = \lambda K^\ast (Ku - b) + u \ \forall u \in \mathcal{H}.
\end{align*}
For $u_1, u_2 \in \mathcal{H}$ it holds
\begin{align*}
\|\nabla g(u_1) - \nabla g(u_2)\| &= \|\lambda K^\ast K (u_1 - u_2) + (u_1 - u_2) \| \leq \left(\lambda \|K\|^2 + 1 \right) \|u_1 - u_2\| \\ &\leq \left(\frac{\lambda}{4} + 1 \right) \|u_1 - u_2\|, 
\end{align*}  
thus $\nabla g$ is $\left(\frac{\lambda}{4} + 1\right)$-Lipschitz continuous. \\
In this setting, the iterative scheme \eqref{itstep} reads 
\begin{align}\label{AlgRecGr}
u_{n+1} = \beta_n u_n - \gamma_n [\lambda K^\ast (K(\beta_n u_n) - b) + \beta_n u_n] \ \forall n \geq 0, 
\end{align} 
with starting point $u_0 \in L^2((0,1))$ and $(\beta_n)_{n \geq 0} \subseteq (0,1]$, $(\lambda_n)_{n \geq 0} \subseteq (0, 2 - \gamma_n (\lambda + 4)/8 ]$ and $(\gamma_n)_{n \geq 0} \subseteq (0, 8/(\lambda + 4))$ are real sequences fulfilling the assumptions (i), (ii) and (iii), respectively, in Theorem \ref{thm3}. 
In Table 2 we report some numerical results obtained when running the iterative scheme \eqref{AlgRecGr} for different starting points. 

When choosing $f(u) = \frac{1}{2} \int_{(0,1)} u^2(x) dx$ and $g(u) = \frac{\lambda}{2} \int_{\Omega} ((Ku)(x) - b(x))^2 dx$, the gradient of $g$ reads 
\begin{align*}
\nabla g(u) = \lambda K^\ast (Ku - b) \ \forall u \in \mathcal{H},
\end{align*}
while the proximal operator of $f$ is given by 
\begin{align*}
\operatorname{prox}_{\gamma f}(u) = \frac{u}{1+\gamma}, 
\end{align*}
where $\gamma >0$.
For $u_1, u_2 \in \mathcal{H}$ it holds
\begin{align*}
\|\nabla g(u_1) - \nabla g(u_2)\| = \|\lambda K^\ast K (u_1 - u_2) \| \leq \lambda \|K\|^2 \|u_1 - u_2\| \leq \frac{\lambda}{4} \|u_1 - u_2\|, 
\end{align*}  
hence $\nabla g$ is $\frac{\lambda}{4}$-Lipschitz continuous.  
In this setting, the iterative scheme \eqref{itstep} reads 
\begin{align}\label{AlgRec}
u_{n+1} = (1 - \lambda_n) \beta_n u_n + \frac{\lambda_n}{1+\gamma_n} [\beta_n u_n - \gamma_n \lambda K^\ast (Ku_n - f)(\beta_n u_n)] \phantom{xx} \forall n \geq 0, 
\end{align} 
with starting point $u_0 \in L^2((0,1))$ and $(\beta_n)_{n \geq 0} \subseteq (0,1]$, $(\lambda_n)_{n \geq 0} \subseteq (0, 2 - \gamma_n \lambda/8 ]$, and $(\gamma_n)_{n \geq 0} \subseteq (0, 8\lambda)$ are real sequences fulfilling the assumptions (i), (ii) and (iii), respectively, in Theorem \ref{thm3}. In Table 1 we report some numerical results obtained when running the iterative scheme \eqref{AlgRec} for different starting points. 

\begin{center}
	\begin{tabular}{lccccc}
		\hline \hline
		$u_0$ & $f$ & \multicolumn{2}{c}{$\gamma = 1.3$}  & \multicolumn{2}{c}{$\gamma_n = 1.3 - 0.1\cdot(-1)^n$}  \\
		\hline
		&  & Number of it. & CPU time in sec. & Number of it. & CPU time in sec.  \\ \hline
		$x^2/10$ & $x$                    & 13  & 4.09281 & - & $> 600$ \\
		$2^x/16$ & $x$                   & 13 &  248.686 & - & $> 600$ \\
		$\sin(x)$  & $x$             & 7 &  12.9818 & - & $> 600$ \\
		$\cos(x)$ & $x$          & 12 & 10.2537 & - & $> 600$ \\
		$x^2/10$ & $x^2$          & 13 & 24.5828 & - & $> 600$ \\
		$2^x/16$ & $x^2$                & 11 & 25.3676 & - & $> 600$ \\
		$\sin(x)$ & $x^2$                        & 9 & 18.5916 & - & $> 600$ \\
		$\cos(x)$ & $x^2$              & 14 & 89.7744 & - & $> 600$ \\
		$x^2/10$ & $\sin(x)$          & 13 & 16.3641 & - & $> 600$ \\
		$2^x/16$ & $\sin(x)$          & 12 & 30.4143 & - & $> 600$ \\
		$\sin(x)$ & $\sin(x)$          & 5 & 8.41192 & - & $> 600$ \\
		$\cos(x)$ & $\sin(x)$      & 14 & 17.4393 & - & $> 600$ \\
		\hline \hline 
	\end{tabular} 	
\end{center}
Table 1: Numerical performances of the Tikhonov regularized gradient method with constant and variable step sizes and different starting points $u_0$.
\begin{center}
	\begin{tabular}{lccccc}
		\hline \hline
		$u_0$ & $f$ & \multicolumn{2}{c}{$\gamma = 1.3$}  & \multicolumn{2}{c}{$\gamma_n = 1.3 - 0.1\cdot(-1)^n$}  \\
		\hline
		&  & Number of it. & CPU time in sec. & Number of it. & CPU time in sec.  \\ \hline
		$x^2/10$ & $x$            & 11  & 1.333 & 7 & 0.291091 \\
		$2^x/16$ & $x$            & 11 &  25.5198 & 7 & 13.7977 \\
		$\sin(x)$  & $x$          & 11 &  25.3024 & 7 & 12.8907 \\
		$\cos(x)$ & $x$           & 11 & 7.72126 & 7 & 2.83167 \\
		$x^2/10$ & $x^2$          & 10 & 12.9192 & 7 & 5.38062 \\
		$2^x/16$ & $x^2$          & 10 & 22.3199 & 7 & 13.649 \\
		$\sin(x)$ & $x^2$         & 10 & 21.8628 & 7 & 13.3971 \\
		$\cos(x)$ & $x^2$         & 10 & 20.5305 & 7 & 12.5219 \\
		$x^2/10$ & $\sin(x)$      & 11 & 9.79028 & 7 & 3.4632 \\
		$2^x/16$ & $\sin(x)$      & 11 & 27.1717 & 7 & 14.137 \\
		$\sin(x)$ & $\sin(x)$     & 11 & 27.5599 & 7 & 13.958 \\
		$\cos(x)$ & $\sin(x)$     & 11 & 6.65283 & 7 & 2.18238 \\
		\hline \hline 
	\end{tabular} 	
\end{center}
Table 2: Numerical performances of the Tikhonov regularized proximal-gradient method with constant and variable step sizes and different starting points $u_0$.

\textbf{Interpretation.} In the numerical experiments we carried out for both approaches described above we considered as regularization parameter $\lambda := 1$, as Tikhonov regularization sequence
$\beta_0 := \frac{1}{4}$, $\beta_n := 1 - \frac{1}{1 + n}$ for $n \geq 1$, and the constant sequence of relaxation parameters $\lambda_n = 0.9$ for any $n \geq 0$. As stopping criterion we used for both iterative schemes $\|u_{n+1} - u_n\| \leq 10^{-4}$.

Table 1 shows that the Tikhonov regularized gradient method with constant step sizes do not reach the demanded accuracy in a reasonable period of time (in our case 600 sec.), while the variant with variable step size does. Table 2 shows that the Tikhonov regularized proximal-gradient method with variable step sizes outperforms both the variant with constant step size as well as the Tikhonov regularized gradient method \eqref{AlgRecGr} from the point of view of the number of iterations and of the CPU time.

\subsection{A split feasibility problem}\label{subsec42}

Let $\mathcal{H}$ and $\mathcal{G}$ be real Hilbert spaces and $L : \mathcal{H} \to \mathcal{G}$ a bounded linear operator. Let $C$ and $Q$ be nonempty, cconvex and closed subsets of $\mathcal{H}$ and $\mathcal{G}$, respectively. The \textit{split feasibility problem} (SFP) searches for a point with the property 
\begin{align}\label{SFP}
x \in C \text{ and } Lx \in Q.
\end{align}
It was originally introduced by Censor and Elfving \cite{CeEl} for solving inverse problems in the context of phase retrieval, medical image reconstruction and intensity modulated radiation therapy. 

We will use the strongly convergent proximal-gradient algorithm stated in \eqref{itstep}  for solving the (SFP). For this purpose, we note that, provided it has a solution, the problem $\eqref{SFP}$ can be equivalently written as
\begin{align}\label{min1}
\min_{x \in \mathcal{H}} \left\{ \delta_C(x) +  \frac{1}{2} \| Lx - \operatorname{P}_{Q} (Lx) \|^2 \right\}.
\end{align} 
We choose in the framework of Remark \ref{rem2} (i) $f = \delta_C$ and $g(x) = \frac{1}{2} \| Lx - \operatorname{P}_{Q} (Lx) \|^2$. The function $g$ is Fr\'echet differentiable with gradient $\nabla g = L^\ast \circ (\operatorname{Id} - P_Q) \circ L$ and it holds for $x_1, x_2 \in \mathcal{H}$
\begin{align*}
\|\nabla g(x_1) - \nabla g(x_2)\| \leq \|L\|^2 \|x_1 - x_2\| 
\end{align*}
hence $\nabla g$ is Lipschitz continuous with constant $\|L\|^2$. The iterative scheme in \eqref{itstep} applied to problem \eqref{min1} reads

\begin{equation}\label{Alg}
(\forall n \geq 0)~\left\{ \begin{array}{rl}
v_n &= (L^\ast \circ (\operatorname{Id} - P_Q) \circ L)(\beta_n x_n) \\
x_{n+1} &= (1-\lambda_n) \beta_n x_n + \lambda_n P_C(\beta_n x_n - \gamma_n v_n).
\end{array}
\right.
\end{equation}

For the numerical experiments we considered  
$$\mathcal{H} = \mathcal{G} = L^2([0,2\pi]) := \left\{ f : [0,2\pi] \to \mathbb{R} : \int_{0}^{2\pi} |f(s)|^2 ds < + \infty \right\}$$ equipped with the scalar product $\langle f, g \rangle := \int_{0}^{2\pi} f(s)g(s) ds$ and the associated norm $\|f\| := \left( \int_{0}^{2\pi} |f(s)|^2 ds \right)^{1/2}$ for all $f,g \in L^2([0,2\pi])$, and the nonempty, convex and closed sets
\begin{align*}
C := \left\{ x \in L^2([0,2\pi]) : \int_{0}^{2\pi} x(s) ds \leq 1 \right\}, 
\end{align*}
and
\begin{align*}
Q := \mathbb{R}_+ v, \text{ with } v(t) := t^2.
\end{align*}
Further, we considered the bounded self-adjoint linear operator $L : L^2([0,2\pi]) \to L^2([0,2\pi])$ 
\begin{align*}
(Lx)(t) := \frac{3t}{8 \pi^{3}} \int_{0}^{2\pi} s \cdot x(s) ~ ds. 
\end{align*}
 For the norm of $L$ we have the following estimate
\begin{align*}
\|L\|^2 \leq \frac{9}{64 \pi^6} \int_{0}^{2 \pi} \int_{0}^{2 \pi} (s \cdot t)^2 ~ ds~dt = 1.
\end{align*}
The projection onto the set $C$ and can be calculated for any $x \in \mathcal{H}$ as (\cite[Example~28.16]{BC}) \[P_C(x) = \left\{\begin{array}{ll} \frac{1 - \int_{0}^{2\pi} x(s) ds}{2\pi} + x, & \mbox{if} \ \int_{0}^{2\pi} x(s) ds > 1, \\[3pt]
x, & \text{else} \end{array}\right. .\] 
On the other hand $P_Q$ is given for any $x \in \mathcal{H}$ by (\cite[Example~28.24]{BC})
\[(P_Q(x))(t) = \left\{\begin{array}{ll} \frac{5 \int_{0}^{2 \pi} s^2 \cdot x(s) ds }{32 \pi^5} \cdot t^2, & \mbox{if} \ \int_{0}^{2\pi} s^2 \cdot x(s) ds > 0 \\[3pt]
0, & \text{else} \end{array}\right. .\]
Note that, since $x = 0$ is a solution to the (SFP), we have $C \cap L^{-1}(Q) \neq \emptyset$.  

\begin{center}
	\begin{tabular}{lcccc}
		\hline \hline
		$x_0$   & \multicolumn{2}{c}{$\gamma_n = 0.5$}  & \multicolumn{2}{c}{$\gamma_n = 1 - \frac{0.5}{1+n}$}  \\
		\hline
		&  Number of it. & CPU time in sec. & Number of it. & CPU time in sec.  \\ \hline
		$t$                     & 8  & 0.292133 & 6 & 0.197763 \\
		$t^2$                     & 12 &  1.74567 & 8 & 0.3915235 \\
		$t^3$                & 17 &  842.058 & 10 & 0.655565 \\
		$\sin(t)$           & 3 & 0.473435 & 2 & 0.28365 \\
		$\cos(t)$            & 1 & 0.050345 & 1 & 0.050911 \\
		$\exp(t)$                  & 19 & 836.925 & 11 & 2.90818 \\
		$\log(t)$                          & 5 & 16.0949 & 4 & 12.8893 \\
		$\sqrt{t}$                & 6 & 2.72155 & 5 & 2.24045 \\
		\hline \hline 
	\end{tabular} 	
\end{center}
Table 3: Numerical performances of the iterative scheme \eqref{Alg} with constant and variable step sizes, different stating points and the constant sequence of relaxation parameters $\lambda_n = 0.4$ for any $n \geq 0$.

\begin{center}
	\begin{tabular}{lcccc}
		\hline \hline
		$x_0$   & \multicolumn{2}{c}{$\gamma_n = 0.5$}  & \multicolumn{2}{c}{$\gamma_n = 1 - \frac{0.5}{1+n}$}  \\
		\hline
		&  Number of it. & CPU time in sec. & Number of it. & CPU time in sec.  \\ \hline
		$t$                     & 4  & 0.123398 & 3 & 0.088829 \\
		$t^2$                     & 6 &  0.251096 & 4 & 0.158801 \\
		$t^3$                & 9 & 0.481185 & 5 & 0.204154 \\
		$\sin(t)$           & 4 & 0.561261 & 3 & 0.412813 \\
		$\cos(t)$            & 1 & 0.048766 & 1 & 0.048814 \\
		$\exp(t)$                  & 10 & 2.14205 & 6 & 1.11041 \\
		$\log(t)$                          & 3 & 10.0328 & 3 & 10.2823 \\
		$\sqrt{t}$                & 3 & 1.4994 & 3 & 1.49998 \\
		\hline \hline 
	\end{tabular} 	
\end{center}
Table 4: Numerical performances of the iterative scheme \eqref{Alg} with constant and variable step sizes, different stating points and the sequence of relaxation parameters  $\lambda_n = \frac{1}{2} + \frac{1}{2+n}$ for any $n \geq 0$.\\

\textbf{Interpretation.} We implemented the iterative scheme \eqref{Alg} in Mathematica using symbolic computation. We considered as Tikhonov regularization sequence $\beta_0 := \frac{1}{4}$, $\beta_n := 1 - \frac{1}{1 + n}$ for $n \geq 1$ and different choices for the sequences of relaxation variables and for the step sizes. As stopping criterion we used 
$$ \frac{1}{2} ||P_C(x_n) - x_n||^2 + \frac{1}{2} ||P_Q(Lx_n) - Lx_n||^2 \leq 10^{-3}.$$ 
In Table 3 and Table 4 we compare the numerical performances of \eqref{Alg} with constant and variable step sizes for different starting points and for two different settings of relaxation variables. One can notice that, in both settings, the variant of the iterative scheme \eqref{Alg} with variable step sizes outperforms the one with constant step sizes from the point of view of the number of iterations and of the CPU time.

\end{document}